\documentclass[11pt,reqno]{amsart}
\usepackage{amsbsy,amsfonts,amsmath,amssymb,amscd,amsthm}
\usepackage{mathrsfs,float,enumitem,ifpdf}
\usepackage[T1]{fontenc}
\usepackage[mathcal]{euscript}
\usepackage{subfigure,enumitem,graphicx,epstopdf}
\usepackage[a4paper,text={6.1in,8in},centering]{geometry}
\usepackage{pxfonts,epstopdf}
\usepackage{pstricks}
\usepackage[colorlinks=true,linkcolor=blue,citecolor=green]{hyperref}
\usepackage{srcltx}
\usepackage[margin=20pt,font=small,labelfont=bf,textfont=it]{caption}

\hypersetup{linktocpage}

\small\normalsize
\theoremstyle{plain}
\newtheorem{mainthm}{Theorem}

\newtheorem{thm}{Theorem}[section]

\newtheorem{lem}[thm]{Lemma}

\newtheorem{prop}[thm]{Proposition}

\theoremstyle{definition}

\newtheorem{rem}[thm]{Remark}

\newcommand{\RR}{{\mathbb{R}}}

\newcommand{\eqdef}{\stackrel{\scriptscriptstyle\rm def}{=}}
\DeclareMathOperator{\Fix}{Fix}

\let\oldtocsection=\tocsection

\let\oldtocsubsection=\tocsubsection
\renewcommand{\tocsection}[2]{\hspace{1.1em}\bf\oldtocsection{#1}{#2}}
\renewcommand{\tocsubsection}[2]{\hspace{1.8em}\oldtocsubsection{#1}{#2}}
\let\oldtocsubsubsection=\tocsubsubsection
\renewcommand{\tocsubsubsection}[2]{\hspace{4.2em}\it\oldtocsubsubsection{#1}{#2}}

\begin{document}
\title[Chaos near a reversible homoclinic bifocus]{
Chaos near a reversible homoclinic bifocus}
\author[Barrientos]{Pablo G. Barrientos}
\address{\centerline{Instituto de Matem\'atica e Estat\'istica, UFF}
   \centerline{Rua Prof. Marcos Waldemar de Freitas Reis, s/n, Niter\'oi,
   Brazil}}
\email{pgbarrientos@id.uff.br} \email{artem@mat.uff.br}

\author[Raibekas]{Artem Raibekas}
\author[Rodrigues]{Alexandre A. P. Rodrigues}
\address{\centerline{Centro de Matem\'atica da Universidade do Porto}
\centerline{Faculdade de Ci\^encias da Universidade do Porto}
\centerline{Rua do Campo Alegre 687, 4169--007 Porto, Portugal}}
\email{alexandre.rodrigues@fc.up.pt}


\begin{abstract}
We show that any neighborhood of a non-degenerate reversible
bifocal homoclinic orbit contains chaotic suspended invariant sets
on $N$-symbols for all $N\geq 2$. This will be achieved by showing
switching associated with networks of secondary homoclinic orbits.
We also prove the existence of \emph{super-homoclinic} orbits
(trajectories homoclinic to a network of homoclinic orbits), whose
presence leads to a particularly rich structure.
\end{abstract}
 \maketitle
\thispagestyle{empty}
\section{Introduction}

A useful indicative of complicated behavior in a differential
equation is the presence of homoclinic orbits to a
hyperbolic equilibrium point. For instance, in any neighborhood of
a Shilnikov homoclinic orbit (a homoclinic orbit connected to a
saddle-focus in dimension three) one can find chaos. This chaotic
behavior is made clear by means of suspended horseshoes in any
number of symbols~\cite{S65,Tresser}. That is,
there are compact invariant hyperbolic sets of the Poincar\'e
return map whose dynamics are topologically conjugate to the Bernoulli
shift on $N$-symbols with $N\geq 2$. Similar results were also
obtained in the
literature~\cite{FS, IbRo, Shilnikov67A,Shilnikov70,Wiggins} for
non-degenerate bifocal homoclinic orbits in the non-resonant case.

Given an ordinary differential equation of the type
\begin{equation} \label{eqdif}
\dot{x}=f(x), \quad x\in \mathbb{R}^4,  \quad f\in C^r, \quad
r\geq 2,
\end{equation}
a \emph{bifocal homoclinic orbit} $\gamma$ is a solution
bi-asymptotic to an equilibrium point $O$  such that $Df(O)$ has a
pair of eigenvalues $\alpha_{k}\pm i\omega _{k}$, $k=1,2$ with
$\alpha_{1}<0<\alpha_{2}$ and $\omega_1\cdot\omega_2\not=0$. This
equilibrium is often called by \emph{bifocus}.  The
\emph{non-resonant} case is given by the \emph{saddle index}
$\delta=|\alpha_1/\alpha_2|\not=1$ and $\gamma$ is said to be
\emph{non-degenerate} if the tangent spaces of the stable and
unstable manifold of $O$  have a trivial intersection (containing
the flow direction) at any point around the connection $\gamma$.
The resonant case $\delta=1$ includes Hamiltonian and reversible
systems.

The presence of chaos
in any neighborhood of a non-degenerate Hamiltonian bifocal
homoclinic orbit was obtained by Devaney in~\cite{Devaney76}. He
proved that, in the critical level set, one can find two-dimensional
suspended horseshoes on any number of symbols arbitrarily close to
the connection. This result was extended by Lerman
\cite{L91,L97,L00} for nearby level sets showing the existence of
infinitely many suspended two-dimensional horseshoes. In
particular,
in any neighborhood of a Hamiltonian bifocal homoclinic orbit
there are compact invariant sets of the Poincar\'e return map,
which are topologically conjugated to the shift
times the identity map
(see~\cite{BIR16}).

System~\eqref{eqdif}
is called \emph{reversible} if there is a linear map
$R:\mathbb{R}^4\to \RR^4$ with $R^2=\mathrm{id}$ and $\dim(\Fix
R)=2$ such that
 $f\circ R = - \, R\circ f$. Recall that
$\Fix R= \{x\in\mathbb{R}^4: Rx=x \}$.
 A trajectory of~\eqref{eqdif} is called
 \emph{reversible} or \emph{symmetric} if it is invariant under the involution $R$.
The similarity between reversible and Hamiltonian systems has been
demonstrated in many cases~\cite{D76}. For instance, both
reversible and Hamiltonian homoclinic orbits are accompanied by a
one-parameter family of periodic orbits~\cite{Devaney77}. However,
nothing is known in general about the presence of chaos in a
neighborhood of a reversible bifocal homoclinic orbit $\gamma$. 
Homburg and Lamb studied in~\cite{HL2006} this situation under the
extra assumption that there is an orbit $\gamma_{a_0}$ in the
one-parameter family $\gamma_{a}$ of accompanying periodic orbits
to $\gamma$ whose stable and unstable manifolds intersect in a
reversible bi-asymptotic orbit $\rho_{a_0}$ to  $\gamma_{a_0}$.
They
concluded that the non-wandering set of the return map describing
the dynamics near $\rho_{a_0}$ is contained in a set with a
lamination of one-dimensional leaves parameterized by a subshift
of finite type, that is similar as in the Hamiltonian case.

In
this paper,
 will prove that in any neighborhood of a non-degenerate reversible bifocal
 homoclinic orbit $\gamma$,
 we find chaos 
  in an arbitrary number of symbols.
  We stress that we do not ask for  additional hypotheses as the one used in~\cite{HL2006}.
  We will say an invariant set $\Lambda$ of a continuous map
  $\Pi$ is \emph{chaotic
on $N$-symbols} if
$\Pi:\Lambda\to \Lambda$ is semi-conjugate to the shift on the set
$\Sigma_N^+=\{1,\dots,N\}^\mathbb{N}$ of sequences of $N$-symbols
and the periodic orbits of $\Pi$ are dense in~$\Lambda$.
This notion of a chaotic set is an extension of the so-called
\emph{chaos in the sense of Block and
Coppel}~\cite{BC92,AK01,MPZ09}. As a consequence of the
semi-conjugation we have that $\Pi|_{\Lambda}$ has positive
entropy, namely $h_{top}(\Pi|_{\Lambda})\geq \log N$
and has sensitive dependence on the initial conditions.

\begin{mainthm} \label{thmA}
For every $N\geq 2$, given any small tubular neighborhood $\mathcal{T}$ of a reversible
non-degenerate bifocal homoclinic orbit $\gamma$ of~\eqref{eqdif}, there is a first return map $\Pi$
describing the dynamics near $\gamma$ on $\mathcal{T}$ which
 has an invariant set $\Lambda_N$ chaotic on $N$-symbols.
\end{mainthm}

Devaney showed in~\cite{Devaney76} that in any tubular
neighborhood of a non-degenerate Hamiltonian bifocal homoclinic
orbit there are infinitely many secondary homoclinic orbits. That
is, other bifocal homoclinic orbits which make several excursions
along the primary homoclinic orbit. This result was extended for
non-degenerate reversible bifocal homoclinics by
H\"arterich~\cite{Hart} showing that the secondary homoclinics are
also reversible and non-degenerate. Thus, in any neighborhood of
the primary orbit, we can find a \emph{homoclinic network}
$$\Gamma_N=\gamma_1\cup\dots\cup \gamma_N$$ associated
with the bifocus  $O$ composed of $N$ different non-degenerate
reversible homoclinic orbits to $O$.
The chaotic sets $\Lambda_N$ of Theorem~\ref{thmA} will be
obtained by showing the occurrence of \emph{switching} with
respect to the network $\Gamma_N$. In other words, we will prove
that each sequence associated to the homoclinics of the network is
shadowed by a nearby trajectory.

To be precise, consider cross-sections $S_i$ transverse to
$\gamma_i$ and write $\Pi$ for the first return map on the
collection of cross-sections $S=S_1\cup\dots \cup S_N$. Let
$\omega=(\omega_i)_{i\in\mathbb{N}} \in \Sigma_N^+=\{1, \ldots,
N\}^\mathbb{N}$.
According
to~\cite{HK2010} (see also \cite{IbRo}), the homoclinic network
$\Gamma_N$ is called \textit{switching} if for each tubular
neighborhood $\mathcal{U}$ of $\Gamma_N$, there exists a flow
trajectory $\mathcal{\tau}\subset \mathcal{U}$ and a point
$x_\omega\in \mathcal{\tau}$ such that $\Pi^i(x_\omega)\in
S_{\omega_{i+1}}$ for all $i\geq 0$. We call $x_\omega$ the
starting point of the realization $\mathcal{\tau}$ of $\omega$.

Following~\cite{ST97}, a trajectory $\mathcal{\tau}$ is said to be
a \emph{super-homoclinic orbit} to the network $\Gamma_N$ if
$\mathcal{\tau}$ is a bi-asymptotic connection to $\Gamma_N$, that
is it accumulates on the network $\Gamma_N$ in forward and
backward time. A super-homoclinic orbit is not a homoclinic loop
itself, but its presence implies the existence of large number of
the so-called
multi-pulse homoclinic loops. 
We say that the homoclinic network $\Gamma_N$ exhibits
\emph{symmetric super-homoclinic switching} if for each $\omega\in
\Sigma_N^+$ we find a super-homoclinic orbit realizing $\omega$
with starting point $x_\omega\in \mathrm{Fix}(R)$. Since
$x_\omega$ belongs to $\Fix(R)$, then the super-homoclinic orbit
also follows the sequence $\omega$ in backward time. When any
prescribed finite path is realized by a reversible homoclinic
(resp.~periodic) orbit starting in $\mathrm{Fix}(R)$, we say that
$\Gamma_N$ exhibits \emph{symmetric homoclinic}
(resp.~\emph{periodic}) \emph{switching}.

In the next main result, we give a
description of the orbits in a neighborhood of the non-degenerate
reversible homoclinic orbit $\gamma$:

\begin{mainthm} \label{thmB} Under the assumption of Theorem~\ref{thmA},
every homoclinic network $\Gamma_N$ in the tubular neighborhood
$\mathcal{T}$ of \eqref{eqdif} exhibits symmetric
super-homoclinic, homoclinic and periodic switching.
\end{mainthm}

The paper is organized as follows. First, in Section~\ref{sec2} we
construct  the Poincar\'e return map through the composition of
local and semi-global maps.  We strongly use the notion of
reversibility. After that, in Section~\ref{sec3}, the spiralling
geometry near the bifocus equilibrium is described. In particular,
we state that a disc transverse to the stable manifold of $O$ is
mapped, under the first return map, into a spiralling sheet. The
geometrical setting is explored in Section~\ref{sec_nova} to
obtain regions on the disk that are mapped by the flow into new
spiralling regions containing discs in a position similar to the
first one. This allows us to establish the recurrence needed for
Section~\ref{final}, where Theorem~\ref{thmA} and
Theorem~\ref{thmB} are proven.


\subsection*{Open questions}
We finish this section with a couple of open questions. Notice that all the periodic and
super-homoclinic orbits found in Theorem~\ref{thmA} and~\ref{thmB}
are symmetric. Thus the first question, less ambitious, but still interesting
is the following.

\textbf{Question 1:} Are there non-symmetric periodic or homoclinic solutions near
$\gamma$?

\noindent Since we are only able to show a semi-conjugacy with the
one-sided shift, a natural problem~is:

\textbf{Question 2:} Is there a semi-conjugacy with the full two-sided shift on $N$-symbols?


\noindent There are several examples in the literature where
switching is proven through the existence of a suspended horseshoe
in any neighborhood of the network~\cite{HK2010,IbRo}. On the
contrary in~\cite{HK2010, Rodrigues_BSPM}, switching has been
shown to result from an attracting network without a suspended
horseshoe. From the constructions in Theorem~\ref{thmA}, the
chaotic sets we obtain do not seem to come from suspended
horseshoes. But they may arise from the suspension of the
so-called partially hyperbolic sets, such as the product of the
horseshoe by the identity map. Moreover, in \cite{Hart} it is
proven that homoclinic orbits are actually accumulated by
\textit{one-parameter families} of periodic orbits. This parameter
can play the role of the center direction of the partially
hyperbolic set. In~\cite{HL2006}, under extra assumptions,
partially hyperbolic sets were shown to exist in reversible
systems. Thus our last question is:

\textbf{Question 3:} Can the semi-conjugacy obtained in Theorem \ref{thmA} be extended to a conjugacy with a partially hyperbolic set?

\section{Setting and return maps (local and global)} \label{sec2}
\subsection{Setting}
Our object of study is the dynamics around a reversible (not
necessarily conservative) non-degenerate bifocal homoclinic orbit
for which we give a rigorous description here.  Specifically, we
study a differential equation
 \begin{equation}
 \label{general1}
 \dot{x}=f(x),  \quad x\in \RR^4
\end{equation}
where $ f: \RR^{4} \rightarrow \RR^4$ is a  $C^2$-vector field
whose flow has the following properties:
\begin{enumerate}
 \item[\textbf{(P1)}]\label{P1} The vector field $f$ is $R$-reversible.
 That is, $$
 R^2=\mathrm{id}, \quad
\mathrm{dim}(\Fix R)=2 \quad \text{ and } \quad  f\circ R
=-R\circ
 f.$$
 In particular, if $x(t)$ is a solution of~\eqref{general1}, then so is $Rx(t)$.
  \medbreak
 \item[\textbf{(P2)}]\label{P3} The origin $O\in \Fix R$
 is a bifocus equilibrium. In the reversible case, it follows
 that the eigenvalues of $Df(O)$ are  $\pm \alpha \pm i\omega$ where
   $\alpha>0$ and $\omega>0$.
 \medbreak
\item[\textbf{(P3)}]\label{P2} There is a non-degenerate
reversible homoclinic orbit $\gamma$ to $O$. That is,
$$
R(\gamma)=\gamma \quad \text{and} \quad
 \dim \left(T_{\gamma(t)}W^u(O) \cap T_{\gamma(t)}W^s(O)\right)=1
 \quad \text{for all $t\in \mathbb{R}$}.
$$
In particular, the reversibility implies that 
$\gamma\cap \Fix(R)\not=\emptyset$.
\end{enumerate}

From now on,  $\varphi(t,x)$ denotes the flow of~\eqref{general1} at time
$t\in \mathbb{R}$ and initial condition $x\in\mathbb{R}^4$. Under assumptions \textbf{(P1)--(P3)}, H\"arterich showed in~\cite{Hart} the
following result:

\begin{thm}[\cite{Hart}] \label{thmHart}
In any tubular neighborhood of $\gamma$ there exist infinitely
many reversible non-degenerate homoclinic orbits to $O$. Moreover,
each homoclinic orbit is accumulated by a one-parameter family of
reversible periodic orbits.
\end{thm}

We fix a tubular neighborhood $\mathcal{T}$ of $\gamma$ and $N\geq
2$. According to Theorem~\ref{thmHart} we can find $N$ different
reversible non-degenerate homoclinic orbits $\gamma_i$ in
$\mathcal{T}$. Consider the network
$$
  \Gamma_N = \gamma_1\cup \dots \cup \gamma_N.
$$
We perform a similar analysis to that of \cite{Hart,IbRo} and study the
first return map over a
cross-section transverse to the homoclinic network $\Gamma_N$. The
flow in a neighborhood of $\gamma$ consists of local and global dynamics.
When a trajectory is near the equilibrium its behavior is
essentially governed by the linearized vector field. Far from the
equilibrium, we use the Tubular Flow Theorem \cite{PM} to create a global return map.

\subsection{Local dynamics}
According to \cite{Hart}, the equation~\eqref{general1} can be linearized near
$O$. That is, there exists a neighborhood $V_O$ of $O$ such that, if $(x_1, x_2, x_3, x_4)$ is in $V_O$, the system
(\ref{general1}) is $C^1$-orbitally equivalent to the linear
system
\begin{equation} \label{linearization1}
\left\{
\begin{aligned}
 \dot{x}_1 &= -\alpha x_1 - \omega x_2,  \\
 \dot{x}_2 &= \omega x_1 - \alpha x_2, \\
 \dot{x}_3 &= \alpha x_3 + \omega x_4,  \\
 \dot{x}_4 &= -\omega x_3 +\alpha x_4.
\end{aligned}
\right.
\end{equation}
\begin{rem} The idea behind \cite{Hart} to linearize the vector field near the origin is  Belitskii's theorem \cite{Belitskii73}, which guarantees the $C^1-$ linearisation, being enough for our purposes. The refinement of this study (as the stability of elliptic fixed points and more involving bifurcations) requires the use of the normal form described in \cite{SSTC}.
\end{rem}

Using \eqref{linearization1}, we may define new bipolar coordinates $(r_s, \phi_s,r_u, \phi_u)$ near the
bifocus $O$:
$$
x_1=r_s \cos(\phi_s), \qquad x_2=r_s \sin(\phi_s) \qquad x_3=r_u
\cos(\phi_u)\qquad \text{and} \qquad x_4=r_u \sin(\phi_u).
$$
The local invariant manifolds are given by
$$
W^s_{loc}(O)=\{ r_u=0\} \qquad \text{and} \qquad   W^u_{loc}(O)=\{
r_s=0\}.
$$
Near $O$, in bipolar coordinates $(r_s, \phi_s,r_u, \phi_u)$, the
dynamics is governed by the differential equations
\begin{equation}
\label{equation1} \dot{r}_s=-\alpha r_s, \qquad
\dot{\phi}_s=\omega, \qquad \dot{r}_u=\alpha r_u \qquad \text{and}
\qquad  \dot{\phi}_u=-\omega
\end{equation}
whose explicit solutions can be written as
$$
r_s(t)=r_s(0) e^{-\alpha t}\qquad\phi_s(t)=\phi_s(0)+ \omega t
\qquad r_u(t)=r_u(0) e^{\alpha t} \qquad   \phi_u(t)=\phi_u(0)-
\omega t.
$$

\subsubsection{Cross sections near to the bifocus equilibrium}
Without restriction we can assume that the linear involution $R$
in \eqref{P1} is given by $R(x_1,x_2,x_3,x_4)=(x_3,x_4,x_1,x_2)$, as in \cite{Hart}. %
Thus the two-dimensional set of fixed points by $R$ is written
in bipolar coordinates as $$\Fix(R) =\{r_s=r_u, \phi_s=\phi_u\}.$$

In order to construct the Poincar\'e map, we consider the
three-dimensional return sections near the origin, $\Sigma^{in}_O$
and $\Sigma^{out}_O$ in $V_O$, which are solid tori defined by
$$
\Sigma^{in}_O=\{r_s=r\} \quad \text{and} \quad \Sigma^{out}_O=
\{r_u=r\}=R(\Sigma^{in}_O)
$$
where $r>0$ is chosen sufficiently small such that
$$
\gamma_i\cap \Sigma_O^{in} =\{q^s_i\} \subset W^s_{loc}(O) \quad
\text{for all $i=1,\dots, N$}.
$$
By reversibility we also have that $\gamma_i\cap \Sigma_O^{out}
=\{q^u_i\} \subset W^u_{loc}(O)$ for all $i=1,\dots, N$.
Trajectories starting   at $\Sigma^{in}_O$ and
$\Sigma^{out}_O$ go outside of $V_O$ in negative and positive
time, respectively.  For convenience we write $(\phi_s^{in}, r_u^{in},
\phi_u^{in})$ and $(r_s^{out}, \phi_s^{out}, \phi_u^{out})$ for
the coordinates in $\Sigma_O^{in}$ and $\Sigma_O^{out}$
respectively. When there is no risk of ambiguity, we write $W^{u/s}$ instead of $W^{u/s}_{loc}(O)$.
%

\subsubsection{Local flow near to the bifocus}
The time of flight inside $V_O$
of a trajectory with initial condition $(\phi_s^{in}, r_u^{in},
\phi_u^{in}) \in \Sigma^{in}_O\backslash W^s_{loc}(O)$ is given by
$$
\frac{1}{\alpha}\ln\left(\frac{1}{r_u^{in}}\right)=
-\frac{\ln(r_u^{in})}{\alpha}.
$$ Integrating (\ref{equation1}), see computations in \cite{Hart,
IbRo}, we may define the map $ \Pi_{O}: \Sigma^{in}_O \backslash W^s(O)
\rightarrow \Sigma^{out}_O $ defined by
\begin{equation}
\label{local_flow_eq} (\phi_s^{in}, r_u^{in}, \phi_u^{in}) \mapsto
(r_s^{out},\, \phi_s^{out},\, \phi_u^{out})=\left(r_u^{in},\,
\phi_s^{in}+\frac{\omega}{\alpha} \ln (r_u^{in}),\,
\phi_u^{in}-\frac{\omega}{\alpha} \ln (r_u^{in})\right).
\end{equation}

\subsection{Transitions and global maps} \label{transitions}

\begin{figure}
 ~\vspace{-0.5cm}
\begin{center}
\ifpdf\includegraphics[scale=0.85]{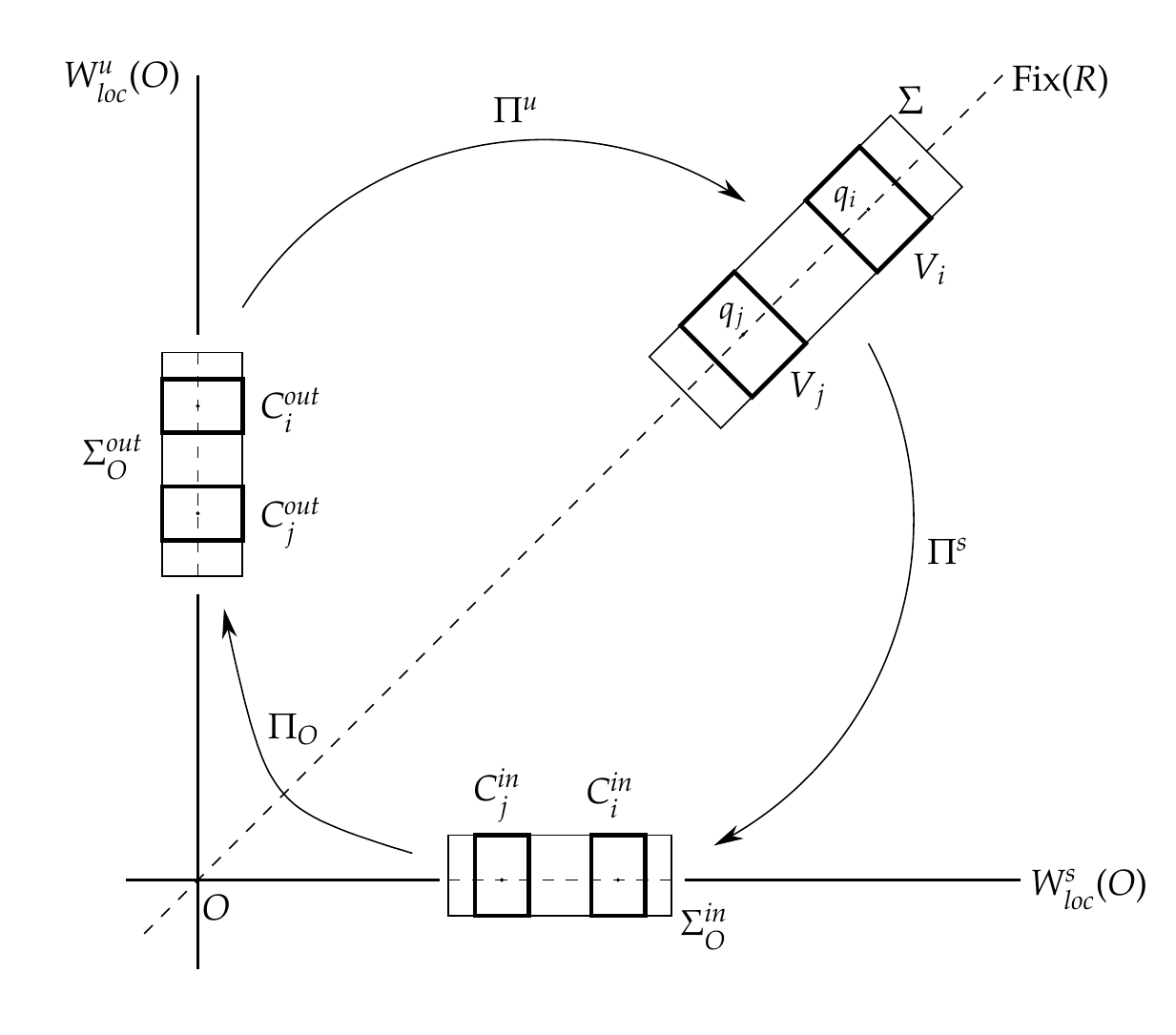} \else
\vspace{0.5cm}\psscalebox{0.85}{\input{fig/retorno}} \fi
\end{center}
\caption{Scheme of the first return map $\Pi= \Pi^u\circ \Pi_O
\circ \Pi^s$ on the section $\Sigma$. } \label{fig1}
\end{figure}

Following Harterich \cite{Hart}, we can make use of the
reversibility in the construction of the global Poincar\'e map by
taking a $R$-invariant cross-section $\Sigma$ containing the
points
$$ \gamma_i \cap \Fix(R)=\{q_i\} \quad \text{for i=1,\dots,N.}$$

A global map $\Pi^u$ between $\Sigma^{out}_O$ and $\Sigma$ is
induced by the flow along $\gamma_i$. To be more specific, given any neighborhood
$V_i$ of $q_i$ in $\Sigma$, there exists a neighborhood $C_i^{out}
\subset \Sigma^{out}_O$ of $q^u_i$ and a $C^2$-map $\tau_i:
C_i^{out} \to \mathbb{R}$ such that
$$
\varphi(\tau_i(q^u_i),q^u_i)=q_i \quad \text{and} \quad
\Pi^u_i(x):=\varphi\left(\tau_i(x),x\right)\in V_i \quad \text{for all $x\in
C^{out}_i$ and $i=1,\dots, N$.}
$$

By taking the sets $V_i\subset \Sigma$  small enough
we can obtain that $C_i^{out}$
 are pairwise disjoint compact neighborhoods of
$q^u_i$ in $\Sigma^{out}_O$ for all $i=1,\dots,N$. Moreover, we can also
take $V_i$ such that $R(V_i)=V_i$. Hence,
$\Pi^u$ is defined as
$$\Pi^u|_{C^{out}_i}=\Pi^u_i \quad \text{for all $i=1,\dots,N$}.$$
By reversibility, $\Pi^s$ is a semi-global map between $\Sigma$
and $\Sigma^{in}_O$ by means of
$$\Pi^s=R\circ (\Pi^u)^{-1}\circ
R.
$$

Finally, we introduce the Poincar\'e first return map on $\Sigma$
following the
homoclinic network $\Gamma_N$ 
as $\Pi= \Pi^u\circ \Pi_O \circ \Pi^s$ (see Figure~\ref{fig1}).
Observe that $\Pi$ is actually defined as the composition of three
maps: first $\Pi^s$ which is well defined from $V=R(V) \subset
\Sigma$ to $\Sigma^{in}_O$, then $$\Pi_O: \Sigma_O^{in}\setminus
W^s_{loc}(O) \to \Sigma^{out}_O$$ and finally $\Pi^u : C \to
\Sigma$ where $V=V_1\cup \dots \cup V_N \subset \Sigma$ and
$C=C_1^{out}\cup \dots \cup C_N^{out} \subset \Sigma^{out}_O$.
Note that $\Pi$ is of class $C^r$ with $r\geq 2$ and it is a
reversible map, i.e., $R\circ \Pi \circ R = \Pi^{-1}$.

\section{Spiralling geometry} \label{sec3}
In this section we describe the strong spiralling behavior of
solutions near $O$. We suggest the reader follows
Figure~\ref{geometry}.

Let $a \in \RR$, $D$ be a disc centered at $p\in\RR^2$. A
\emph{spiral} on $D$ around the point $p$ is a smooth curve $S
:[0, \infty) \rightarrow D,$ satisfying $ \lim_{s\to \infty } S
(s)=p$ and such that if $S (s)=(r(s),\phi(s))$ is its expression
in polar coordinates around $p$ then
\begin{enumerate}
\item $r(s)$ is bounded by two monotonically decreasing maps converging to zero as $s \rightarrow
\infty$,
\item $\phi(s)$ is monotonic for some unbounded subinterval of $[0,\infty)$
and
\item $\lim_{s\to +\infty}|\phi(s)|=\infty$.
\end{enumerate}
The notion of spiral may be naturally extended to any set
diffeomorphic to a disc.

\begin{figure}[ht]
\begin{center}
\includegraphics[scale=0.9]{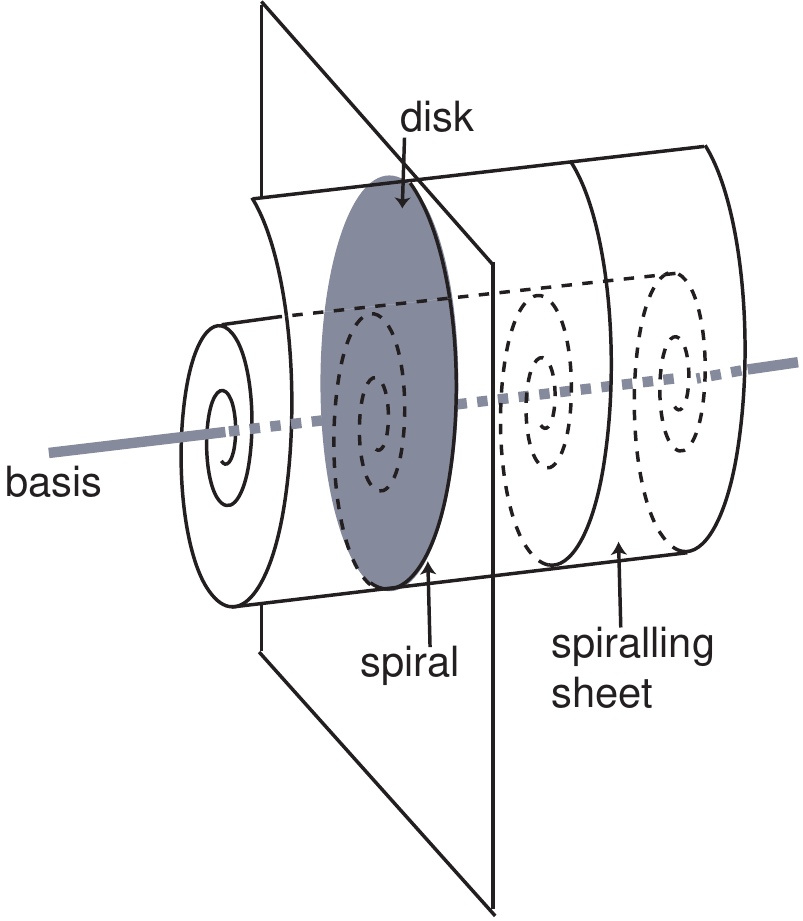}
\label{geometry}
\end{center}
\caption{Geometric structures. } \label{geometry}
\end{figure}

A two-dimensional manifold  $\mathcal{S}$ embedded in
$\mathbb{R}^3$ is called a \emph{spiralling sheet} accumulating on
a curve $\mathcal{C}$ if there exist a spiral $S$ around $(0,0)$,
a neighborhood $V\subset \mathbb{R}^3$ of $\mathcal{C}$, a
neighborhood $W_0 \subset \mathbb{R}^2$ of the origin, a
non-degenerate closed interval $I$ and a diffeomorphism $\eta:
V\to I \times W_0$ such that
$$
\eta(\mathcal{S}\cap V)=I \times (S \cap W_0) \quad \text{and}
\quad \mathcal{C}=\eta^{-1}(I\times \{0\}).
$$
The curve $\mathcal{C}$ may be called the \emph{basis} of the
spiralling sheet. Up to a diffeomorphism, we may think on a
spiralling sheet accumulating on a curve as the cartesian product
of a spiral and a curve.  In the present paper, the curve
$\mathcal{C}$ lies on the invariant manifolds of $O$. Each
cross section to $\mathcal{C}$ intersects the
spiralling sheet $\mathcal{S}$ into a spiral. Note also that the
diffeomorphic image of a spiralling sheet contains a spiralling
sheet.

%
%
The following result will be essential  in the
sequel. It shows that a set diffeomorphic to a disc transverse to
$W^s_{loc}(O)\cap \Sigma^{in}_O$ is sent by
$\Pi_O$~into~a~spiralling~sheet.
\begin{prop}[\cite{Hart, IbRo}]
\label{My_Prop5} For $\nu>0$ arbitrarily small, let $\Xi: D\subset
\mathbb{R}^2 \rightarrow \RR$ be a $C^1$ map defined on the disc
$D=\{(u,v)\in\mathbb{R}^2 : 0\leq u^2+v^2 \leq \nu <1\}$ and let
$$
\mathcal{F}^{in}=\{(\phi_s^{in}, r_u^{in}, \phi_u^{in}) \in
\Sigma^{in}_O : \  \phi_s^{in}=\Xi(r_u^{in} \cos
\phi_u^{in},r_u^{in} \sin \phi_u^{in}),\ 0\leq r_u^{in}\leq \nu,\
0\leq \phi_u^{in} <2\pi\}
$$
Then 
the set $\Pi_{O}(\mathcal{F}^{in}\backslash W^s_{loc}(O))$ is a
spiralling sheet accumulating on $W^u_{loc}(O) \cap
\Sigma^{out}_O$.
%
%
 \end{prop}

\section{The recurrence}
\label{sec_nova} In Section  \ref{sec2},  we have fixed pairwise disjoint
compact neighborhood $V_i$ of $q_i$ in $\Sigma$ for $i=1,\dots,N$.
By means of the global transition maps we have also got pairwise
disjoint compact neighborhoods $C^{out}_i=(\Pi^u)^{-1}(V_i)$ of
$q^u_i$ in $\Sigma^{out}_O$ and $C^{in}_i=\Pi^s\circ
R(V_i)=R(C^{out}_i)$ of $q^s_i$ in $\Sigma^{in}_O$ for
$i=1,\dots,N$. We take discs
$$
 D_i \subset V_i\cap \Fix(R) \subset  \Sigma \ \text{centered at $q_i$ for all
 $i=1,\dots,N$.}
$$
Now, we introduce the local stable un unstable manifolds of $O$ in
$\Sigma$ as
$$
W^u_i = \Pi^u\left(W^u_{loc}(O)\cap C^{out}_i\right) \quad \text{and} \quad
W^s_i = {(\Pi^s)}^{-1}\left(W^s_{loc}(O)\cap C^{in}_i\right) \quad \text{for
$i=1,\dots,N$}.
$$
Observe that the orbit starting in $W^s_i$ (resp.~$W^u_i$) goes directly to $O$, in forward (resp.~backward)
time. The proof of the main results needs the following basic result.

\begin{lem} \label{lem4}
If there is an integer $n\geq 0$  such that $x\in W^s_i\cap
\Pi^n(\Fix(R))$ or $x\in W^u_i\cap \Pi^{-n}(\Fix(R))$, then the
associated solution is a reversible homoclinic orbit to $O$.
\end{lem}
\begin{proof} We prove the case that $x\in W^s_i\cap
\Pi^n(\Fix(R))$. The other case is analogous. We have that
$\Pi^{-n}(x)\in \Fix(R)$ and hence, by the reversibility,
$\Pi^{-2n}(x)\in W^u_i$. Thus the orbit associated with $x$ is a
homoclinic orbit.
\end{proof}
If $\mathcal{R}$ is a measurable set of $ \Fix(R)\cap \mathcal{T}$ let us denote by $\mathcal{A(R)}$ its usual area.

\begin{prop} \label{lem5}
For any $k\geq 1$, $n\in \mathbb{Z}^{k-1}$ and $i\in
\{1,\dots,N\}^k$
there exist pairwise disjoints compact sets
diffeomorphic to a disc
$$
  D_{ij}^{n m}   \subset \Fix(R)   \quad
  \text{for all $m\in\mathbb{Z}$ and $j=1,\dots, N$}
$$
such that
\begin{enumerate}[itemsep=0.1cm]
\item \label{item1} $D^{nm}_{ij} \subset D_{i}^{n}$,
\item \label{item2} $\Pi^{k}(D^{nm}_{ij}) \subset V_{j}$,
\item \label{item3} there is $q^{nm}_{i j} \in D^{n m}_{ij}$
such that $\Pi^{k}(q^{nm}_{ij}) \in W^s_j$ and
\item \label{item4} there is $\lambda<1$ such that
$\mathcal{A}(D^{nm}_{ij}$)<$\lambda\cdot$ $\mathcal{A}(D^{n}_{i}$).
\end{enumerate}

\end{prop}

\begin{proof} As a consequence of Lemma~\ref{lem4} and
since $\gamma_i$ is a non-degenerate homoclinic orbit, the
one-dimensional curves $W^s_i$ and $W^u_i$ are both transverse at
$\{q_i\}=W^s_i\cap W^u_i$ to the two dimensional disc $D_i \subset
\Fix(R)$. Thus $\Pi^s(D_i)$ is a set diffeomorphic to
two-dimensional disc transverse to $W^s_{loc}(O)\cap C^{in}_i
\subset \Sigma_O^{in}$. Now we will proceed inductively.

According to Proposition~\ref{My_Prop5},
$\Pi_O(\Pi^s(D_i)\setminus \{q^s_i\})=\Pi_O\circ
\Pi^s(D_i\setminus \{q_i\})$ is a spiralling sheet accumulating on
$W^u_{loc}(O)\cap \Sigma^{out}_O$. Hence, $
\Pi(D_i\setminus\{q_i\})=\Pi^u\circ \Pi_O\circ
\Pi^s(D_i\setminus\{q_i\})$ contains a spiralling sheet
 in $\Sigma$ accumulating on $W^u_j$ for
$j=1,\dots,N$., which will be denoted by $\mathcal{S}_{ij}$.  See Figure~\ref{fig}.

\begin{figure}
~\vspace{-0.5cm}
\begin{center}
\ifpdf\includegraphics[scale=0.7]{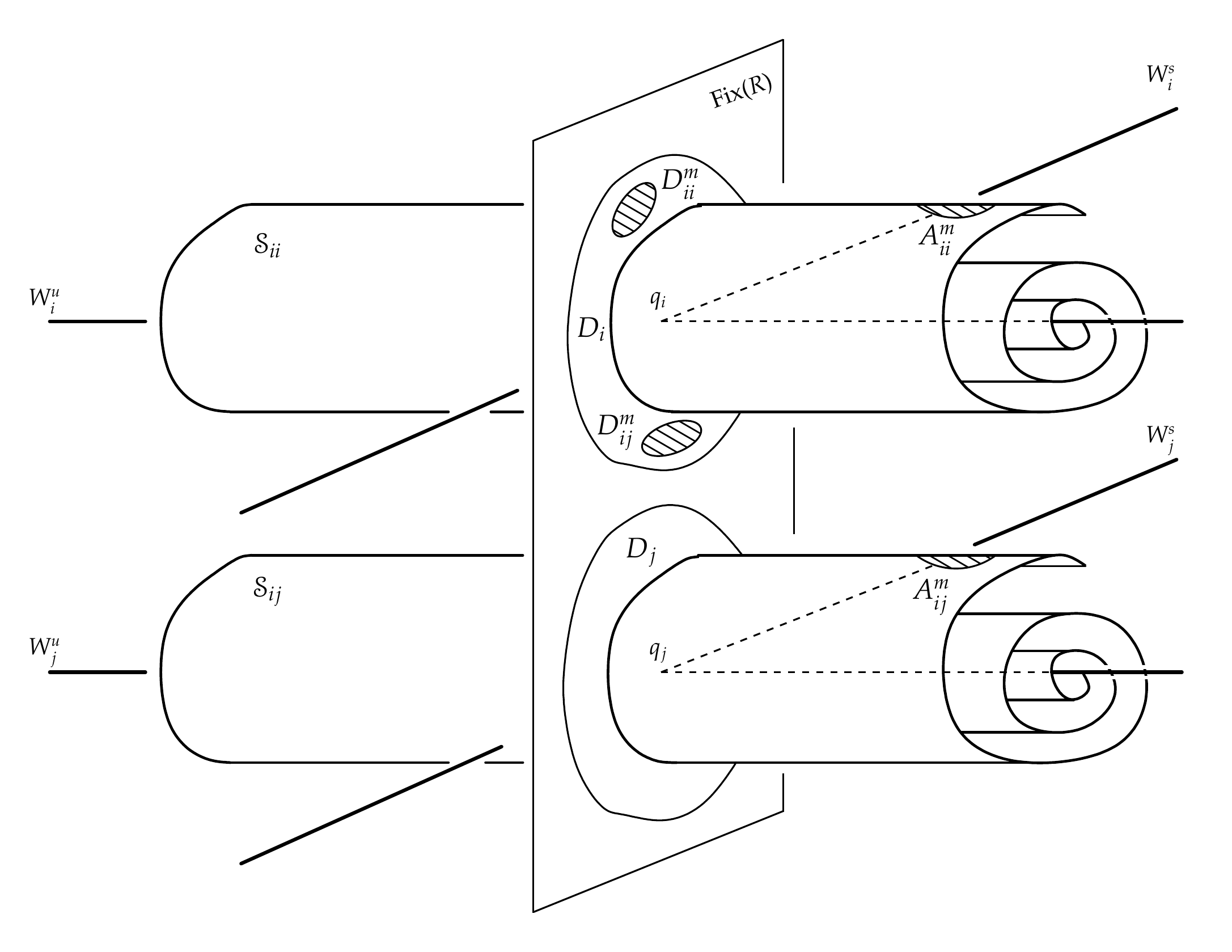} \else
\psscalebox{0.7}{\input{fig/espiral}} \fi
\end{center}
\caption{Spiralling geometry for the return map on section $\Sigma$. } \label{fig}
\end{figure}

From the quasi-transversality between $W^u_j$ and $W^s_j$ (Property \textbf{(P3)}), it
follows that $\mathcal{S}_{ij}$ transversally intersect $W^s_j$
infinitely many times. More precisely, we find a pair of
intersection points in each turn of the spiralling sheet. We
denote by $Q_{ij}^m$ these points in $W^s_j \cap \mathcal{S}_{ij}$
for $m\in \mathbb{Z}$.
From the transversality between $W^s_j$ and $\mathcal{S}_{ij}$ at
$Q^m_{ij}$, around these points in the
spiralling sheet, we can find pairwise disjoint small compact
neighborhoods $A^m_{ij} \subset V_j$, which are diffeomorphic to a disc, for all $m\in\mathbb{Z}$.
Let $$D^m_{ij}=\Pi^{-1}(A_{ij}^m) \subset D_i.$$ Then,
$D^m_{ij}$ are pairwise disjoint compact neighborhoods of
$q_{ij}^m=\Pi^{-1}(Q_{ij}^m)\in D^m_{ij}$. Moreover, by decreasing the Area of  $A^m_{ij}$ we can assume that
$\mathcal{A}(D^{nm}_{ij}$)<$\lambda\cdot$ $\mathcal{A}(D^{n}_{i}$), for some
$\lambda<1$. This proves the result for $k=1$.
Arguing by induction and repeating the above procedure on each set
$\Pi^s\circ\Pi^{k-1}(D^n_{i})$, which up to diffeomorphism is a
two-dimensional disc in $\Sigma^{in}_O$ transverse to
$W^s_{loc}(O)$, we may conclude the properties stated in the lemma.
\end{proof}

Let
\begin{equation}
\label{def_K}
   K=\bigcap_{k\in \mathbb{N}} \bigcup_{n\in\mathbb{Z}^{k-1}}
   \bigcup_{i \in \{1,\dots,N\}^k}  D_{i}^n
   \quad \text{and} \quad
   H = \{q^{n}_{i}: i \in \{1,\dots,N\}^k, n\in \mathbb{Z}^{k-1},  k\in \mathbb{N}\}.
\end{equation}
From Proposition~\ref{lem5} it follows that $K$ is an non-empty set and
$K\subset \overline{H}$.
Moreover,
for any pair of sequences $i=(i_k)_{k\in\mathbb{N}}\in
\Sigma_N^+=\{1,\dots,N\}^\mathbb{N}$ and
$n=(n_k)_{k\in\mathbb{N}}\in \mathbb{Z}^\mathbb{N}$, there is a
unique point $x^n_i \in K$ such that \begin{equation}
\label{eq:superhomolcinic}
  \{ x^n_{i}\} = \bigcap_{k\in\mathbb{N}} D^{n_1\dots
  n_{k-1}}_{i_1\dots i_k}.
\end{equation}
\begin{rem} In fact, any point in $K$ is uniquely
identified as above
by a pair 
$(i,n)\in \Sigma^+_N
\times \mathbb{Z}^\mathbb{N}$. \end{rem}

Again, Proposition~\ref{lem5} and Equation~\eqref{eq:superhomolcinic}
imply that
$$
\Pi^k(x_i^n) \in V_{i_{k+1}} \quad \text{for all $k\geq 0$}.
$$
Since $x^n_i\in \Fix(R)$, then by the reversibility
we also have that $\Pi^{-k}(x_i^n)\in V_{i_{k+1}}$ for all $k\geq
1$.
This proves the following:

\begin{prop} \label{prop:switching}
The homoclinic network $\Gamma_N=\gamma_1\cup\dots\cup \gamma_N$
is switching by reversible trajectories. Moreover, the starting
point of the  orbit realization can be taken in $K\subset
\Fix(R)$ (see \eqref{def_K}).

\end{prop}

The next lemma will prove that the flow orbit associated with the
point in $K$ approaches the homoclinic network $\Gamma_N$ in forward
time (by reversibility this also holds in backward time). To shorten the
notation, we will say that a sequence $(y_n)_n$ converges to a set
$A$ if  the distance between $y_n$ and $A$ goes to zero as $n\to
\infty$.

\begin{lem} \label{lem:superhomolcinic}
If $x\in K$ then $\Pi^k(x)$ converges to $\{q_1,\dots,q_N\}$ as
$|k|\to \infty$.
\end{lem}
\begin{proof}
By reversibility, it suffices to prove the lemma for $k\to \infty$. Let $x\in K$ and
since $K$ is identified with $\Sigma^+_N\times
\mathbb{Z}^\mathbb{N}$, by means of~\eqref{eq:superhomolcinic} we
get the sequences $i=(i_k)_{k\in\mathbb{N}}\in\Sigma_N^+$ and
$n=(n_k)_{k\in\mathbb{N}}\in\mathbb{Z}^\mathbb{N}$ such that
$$
\Pi^{k-1}(x)\in A_k \eqdef \Pi^{k-1}\left(D^{n_1\dots
n_{k-1}}_{i_1\dots i_{k}}\right) \quad \text{for all $k\geq 1$}.
$$

By the inductive construction, $A_k$ is a small
compact neighborhood of $\Pi^k(q_{i_1\dots i_k}^{n_1\dots n_{k-1}})$, which is diffeomorphic to a
two-dimensional disc transverse to $W^s_{i_{k}}$ in a spiralling
sheet accumulating on $W^u_{i_{k}}$. Moreover, we can assume that
the radius $r^{in}_k$ of the disc $\Pi^s(A_{k})$ in
$\Sigma^{in}_O$ goes to zero as $k\to \infty$.
According to~\eqref{local_flow_eq}, the radius of the
spiraling sheet
$$
\mathcal{S}_{k}=\Pi_O\circ \Pi^s(A_{k}\backslash W^s(O))
$$
in $\Sigma^{out}_O$ is $r^{out}_k=r^{in}_k$. Thus, $r^{out}_k\to
0$ as $k\to \infty$ and this sequence of spiralling sheets
converges to $W^u_{loc}(O)$. This implies that the
spiralling sheets in $\Pi^u(\mathcal{S}_{k})$ converge to $W^u_{1}\cup
\dots \cup W^u_N$. Therefore, $A_k$ converges to
$\{q_1,\dots,q_N\}$ as $k\to \infty$ and consequently,
$\Pi^{k-1}(x)$ converges to $\{q_1,\dots q_N\}$.
\end{proof}

\section{Proof of the main results}
\label{final} Let $\mathcal{P}$ be the set of points $x\in V\cap
\Fix(R)$ such that $\Pi^n(x)=x$ for some $n\geq 1$ where
$V=V_1\cup \dots \cup V_N$. Finally, we will prove our main
results:

\begin{proof}[Proof of Theorem~\ref{thmB}]
From Proposition~\ref{prop:switching} and
Lemma~\ref{lem:superhomolcinic} we conclude that $\Gamma_N$
exhibits symmetric super-homoclinic switching. It remains to show
symmetric homoclinic and periodic switching.

According to Lemma~\ref{lem4} and Lemma~\ref{lem5} the flow orbit
associated to any point $p\in H$ is a reversible homoclinic orbit.
For any $k\geq 1$, $n=(n_1,\dots,n_{k-1})\in \mathbb{Z}^{k-1}$ and
$i=(i_1,\dots,i_k)\in \{1,\dots,N\}^k$ we find a point $q=q^n_i
\in H$ such that
$$
 \Pi^j(q) \in V_{i_{|j|+1}}  \quad \text{for
  $|j| \leq k-1$,  \quad}
  \Pi^{k-1}(q)\in W^s_{i_k} \quad \text{and} \quad \Pi^{-(k-1)}(q)\in
  W^u_{i_k}.
$$
Thus, any prescribed finite path is realized by a reversible
homoclinic orbit starting in $\Fix(R)$, and so $\Gamma_N$ exhibits
symmetric homoclinic switching.

Finally, according to Theorem~\ref{thmHart}, each point $q^n_i\in
H$ is accumulated by a one-parameter family $(p_a)_a$ of periodic
points $p_a \in P$.  In particular $H\subset \overline{P}$.
 Moreover, these periodic orbits follow the homoclinic
orbit starting at $q^n_i$ around a turn before closing. Thus
$$
  \Pi^j(p_a) \in V_{i_{|j|+1}} \quad |j|\leq k-1, \quad \text{and} \quad
  \Pi^{k-1}(p_a)=\Pi^{-(k-1)}(p_a),
$$
proving that $\Gamma_N$ exhibits symmetric periodic switching.
\end{proof}

\begin{proof}[Proof of Theorem~\ref{thmA}]
Consider $\Lambda=K\cup  \mathcal{P}$ and let
$$\Lambda_N=\bigcup_{i\in\mathbb{Z}} \Pi^i(\Lambda).$$
Clearly $\Lambda_N$ is a $\Pi$-invariant set. Since $H\subset
\overline{\mathcal{P}}$ and $K\subset \overline{H}$, then $\Lambda\subset
\overline{\mathcal{P}}$, and so the  set of periodic
points is dense in $\Lambda_N$. It remains to prove that
$\Pi|_{\Lambda_N}$ is semi-conjugated to the shift $\sigma$ on the
set $\Sigma^+_N=\{1,\dots,N\}^\mathbb{N}$.

 We denote
by $\pi: \Lambda_N \to \Sigma_N^+$ the coding map defined by
$$
  \pi(x)=(i_n)_{n\in\mathbb{N}} \in \Sigma^+_N \quad
  \text{if} \ \ \Pi^k(x)\in V_{i_{k+1}} \ \ \text{for all
  $k\geq 0$.}
$$
Clearly $\pi$ satisfies that $\pi \circ \Pi= \sigma\circ \pi$.
From Theorem~\ref{thmB}, the network $\Gamma_N$ is switching and
one can take the starting point of the realization in
$K\subset\Lambda_N$. This implies that $\pi$ is onto (surjective). It is left to see
that $\pi$ is also a continuous map.

Fix $x\in \Lambda_N$, $M\in \mathbb{N}$ and for $z\in\Lambda_N$ denote by $\pi(z)=(j_n)_{n\in\mathbb{N}}$.
It is enough to show that there exists $\delta>0$ such that
if $d(z,x)<\delta$, then $j_n=i_n$ for all $n\leq M$.
Now there exists a $\delta_1>0$ such that, if $d(z,x)<\delta_1$ then $z \in V_{i_1}$, and thus
$j_1=i_1$.  Since the return map $\Pi$ is continuous at $x$, we may
find a $\delta_2$ such that if $d(z,x)<\min\{\delta_1,\delta_2\}$
then $d(\Pi(z),\Pi(x))<\delta_1$. As above this implies that
$j_1=i_1$ and $j_2=i_2$. Proceeding inductively until the $n$-th
iterate of $\Pi$ we obtain
$\delta=\min\{\delta_1,\delta_2,\dots,\delta_n\}>0$ as required.
This completes the proof of the semi-conjugation between
$\Pi:\Lambda_N\to \Lambda_N$ and $\sigma:\Sigma^+_N\to\Sigma_N^+$.
\end{proof}

%
%
%

\subsection*{Acknowledgements}
P.~G.~Barrientos was supported by MTM2017-87697-P from Ministerio
de Econom\'ia y Competividad de Espa\~na and CNPQ-Brasil.  A. Rodrigues aknowledges
financial support from Program INVESTIGADOR FCT (IF/00107/2015)
and Centro de Matem\'atica da Universidade do Porto. Centro de
Matem\'atica da Universidade do Porto (CMUP -- UID/MAT/00144/2013)
is funded by FCT (Portugal) with national (MEC) and European
structural funds through the programs FEDER, under the partnership
agreement PT2020. Also, part of this work has been written during the stay of A. Rodrigues in Nizhny Novgorod University, supported by the grant RNF 14-41-00044.

\bibliographystyle{alpha}


\begin{thebibliography}{99}


\bibitem{AK01} B. Aulbach, and B. Kieninger, \emph{On three definitions of chaos}, Nonlinear Dyn. Syst. Theory 1, 2337, 2001





\bibitem{BIR16} P. G. Barrientos, S. Iba\~nez, and J. A. Rodr\'iguez,
\emph{Robust cycles unfolding from conservative bifocal homoclinic
orbits} 31 \textbf{4}, 546--579, 2016


\bibitem{Belitskii73} G. R. Belitskii, \emph{Functional equations and conjugacy of local diffeomorphisms of finite smoothness class},
Funkcional. Anal. i Prilozen, 7, 17--28, 1973
%

\bibitem{BC92} L.S. Block, and W. A. Coppel, \emph{Dynamics in One Dimension}, Springer Lecture Notes, 1513, Springer Verlag, Berlin, 1992




\bibitem{D76} R. L . Devaney, \emph{Reversible Diffeomorphisms and Flows},
Transactions of the American Mathematical Society, 218, 89-113,
1976


\bibitem{Devaney76} R. L. Devaney, \emph{Homoclinic orbits in Hamiltonian systems}, J. Diff. Equations, 21, 431--438, 1976

\bibitem{Devaney77} R. L . Devaney, \emph{Blue sky catastrophes in reversible and Hamiltonian systems}, Ind. Univ. Math. J.,
26,  247--263, 1977


 \bibitem{FS} A.C. Fowler, and C. T. Sparrow,\emph{ Bifocal homoclinic orbits in four dimensions}. Nonlinearity 4,  1159--1182, 1991



\bibitem{Hart} J. Harterich, \emph{Cascades of reversible homoclinic orbits to a bi-focus equilibrium}, Physica D, 112, 187--200, 1998

\bibitem{HL2006} A. J. Homburg, and J. Lamb, \emph{Symmetric homoclinic tangles in reversible systems}, Ergodic Theory Dynam. Systems 26, 1769--1789, 2006

\bibitem{HK2010} A. J. Homburg, and  J. Knobloch, \emph{Switching homoclinic networks}, Dynamical Systems: an International Journal, Vol. 25 (3), 351--358, 2010

\bibitem{IbRo} S. Ib\'a\~nez, and  A.A.P. Rodrigues, \emph{On the dynamics near a homoclinic network to a bifocus:
switching and horseshoes}, Int. J. Bifurcation Chaos 25, 1530030,
2015

















\bibitem{L91}
L. Lerman, \emph{ Complex dynamics and bifurcations in a {H}amiltonian
system having a
  transversal homoclinic orbit to a saddle focus}, Chaos.
1, 174--180, 1991


\bibitem{L97}
L. Lerman, \emph{ Homo- and heteroclinic orbits, hyperbolic subsets in a
one-parameter
  unfolding of a {H}amiltonian system with heteroclinic contour with two
  saddle-foci}, Regul Khaoticheskaya Din. 2, 139--155, 1997.

\bibitem{L00}
L. Lerman, \emph{ Dynamical phenomena near a saddle-focus homoclinic
connection in a
  {H}amiltonian system},  J Statist Phys., 101, 357--372, 2000

\bibitem{MPZ09}  A. Medio, M. Pireddu and F. Zanolin,
\emph{Chaotic dynamics for maps in one and two dimensions: a
geometrical method and applications to economics.} International
Journal of Bifurcation and Chaos 19, 10 3283--3309,  2009.


\bibitem{PM} J. Palis, and W. Melo,
\emph{Geometric Theory of Dynamical Systems : An Introduction},
Springer, 1982

\bibitem{Rodrigues_BSPM} A.A.P. Rodrigues, \emph{Is there switching without suspended horseshoes?}, Bull. Soc. Port. of Maths, 2016

\bibitem{S65}
L. P. Shilnikov, \emph{A case of the existence of a denumerable set of
  periodic motions}, Sov. Math. Dokl. \textbf{6}, 163--166, 1965


\bibitem{Shilnikov67A} L. P. Shilnikov, \emph{The existence of a denumerable set of periodic motions in four dimensional space in an extended neighbourhood of a saddle-focus}, Sovit Math. Dokl., 8(1),  54--58, 1967


\bibitem{Shilnikov70}
  L. P. Shilnikov, \emph{A contribution to the problem of the structure of an extended neighborhood of a rough equilibrium state of saddle-focus type,},  Math. USSR Sb., 10, 91--102, 1970

\bibitem{ST97}
L. P. Shilnikov and D. Turaev, \emph{Super-homoclinic orbits and
multi-pulse homoclinic loops in Hamiltonian systems with discrete
symmetries.} Regul.~Khaoticheskaya Din~2, 126--138, 1997

\bibitem{SSTC} L. P. Shilnikov, A. L. Shilnikov, D. V. Turaev, and
L. O. Chua, \emph{Methods of qualitative theory in nonlinear dynamics (Part 1)},
World Scientific Publishing Co., 1998.


\bibitem{Tresser} C. Tresser, \emph{About some theorems by L. P. Shilnikov}, Ann. Inst. H. Poincar\'e, 40, 441--461, 1984
\bibitem{Wiggins} S. Wiggins, \emph{Introduction in Applied Nonlinear Dynamical Systems and Chaos}, Springer, 1990



\end{thebibliography}

\end{document}